\documentclass{article}

\usepackage[utf8]{inputenc}
\usepackage{a4wide}
\usepackage{color}
\usepackage{graphicx}
\usepackage{xspace}
\usepackage{url}
\usepackage{amsmath,amssymb}
\usepackage{hyperref}
\usepackage{authblk}
\usepackage{comment}

\graphicspath{ {./body/} }

\newcommand{\stef}[1]{{\color{red}\emph{Stéphan:} #1}}
\newcommand{\lv}[1]{#1}

\newcommand{\stb}[1]{#1}
\newcommand{\reviews}[1]{#1}

\newtheorem{theorem}{Theorem}
\newtheorem{conjecture}{Conjecture}
\newtheorem{problem}{Problem}

\newtheorem{corollary}{Corollary}
\newtheorem{proposition}{Proposition}

\def\Box{\hbox{\hskip 1pt \vrule width 4pt height 8pt depth 1.5pt \hskip 1pt}}
\newenvironment{proof}{\medskip\noindent\textbf{Proof.}}{{}\hfill$\Box$\\}
\DeclareMathOperator{\rf}{{\it rf}}

\title{Temporalizing digraphs via linear-size balanced bi-trees}

\author[1]{Stéphane Bessy}
\affil[1]{LIRMM, Univ Montpellier, CNRS, Montpellier, France}
\author[2]{Stéphan Thomassé}
\affil[2]{Univ Lyon, EnsL, UCBL, CNRS, LIP, F-69342, LYON Cedex 07, France}
\author[3]{Laurent Viennot}
\affil[3]{Inria Paris, Irif, Université Paris Cité}

\begin{document}


\maketitle
\begin{abstract}
  In a directed graph $D$ on vertex set $v_1,\dots ,v_n$, a
  \emph{forward arc} is an arc $v_iv_j$ where $i<j$. A pair $v_i,v_j$
  is \emph{forward connected} if there is a directed path from $v_i$
  to $v_j$ consisting of forward arcs. In the {\tt Forward Connected
    Pairs Problem} ({\tt FCPP}), the input is a strongly connected
  digraph $D$, and the output is the maximum number of forward
  connected pairs in some vertex enumeration of $D$. We show that {\tt
    FCPP} is in APX, as one can efficiently enumerate the vertices of
  $D$ in order to achieve a quadratic number of forward connected
  pairs. For this, we construct a linear size balanced bi-tree $T$ (an
  out-tree and an in-tree with same size which roots are
  identified). The existence of such a $T$ was left as an open problem
  motivated by the study of temporal paths in temporal networks. More
  precisely, $T$ can be constructed in quadratic time (in the number
  of vertices) and has size at least $n/3$. The algorithm involves a
  particular depth-first search tree (Left-DFS) of independent
  interest, and shows that every strongly connected directed graph has
  a balanced separator which is a circuit. Remarkably, in the request
  version {\tt RFCPP} of {\tt FCPP}, where the input is a strong
  digraph $D$ and a set of requests $R$ consisting of pairs
  $\{x_i,y_i\}$, there is no constant $c>0$ such that one can always
  find an enumeration realizing $c.|R|$ forward connected pairs
  $\{x_i,y_i\}$ (in either direction).
\end{abstract}
\section{Introduction}
\label{sec:introduction}



Motivated by network design applications, the following problem of scheduling the arcs of a multi-digraph was mentioned as an open problem in \cite{BrunelliCV23} and formally introduced in \cite{BalliuBCOV2023}. The {\tt Maximum Reachability
Edge Temporalisation (MRET)} consists in assigning a time label $t_{xy}\in \mathbb{N}$ to each arc $xy$ of a digraph $D=(V,A)$ so as to maximize the number of pairs $x,z$ of vertices connected by a \emph{temporal path}, that is a path from $x$ to $z$ where time labels strictly increase along the path. 
In the acyclic case, all existing paths can be made temporal using a topological ordering of the vertices, and transferring the index of a node $x$ to every arc leaving $x$. The problem becomes particularly intriguing in the strongly connected case in which every pair of vertices are connected by a path. It was shown in~\cite{BalliuBCOV2023} that \texttt{MRET} is NP-hard even restricted to strongly connected digraphs. In the same paper, the authors suggest that \texttt{MRET} is in APX by conjecturing, in any strongly connected digraph, the existence of some arc-disjoint in-branching $T^-$ and out-branching $T^+$ \reviews{(allowing arbitrary overlapping in terms of vertices)}, both with linear size and rooted at the same vertex $r$. 

We show that this conjecture holds, and since the branchings can be constructed in polynomial time, that \texttt{MRET} is indeed in APX \stb{in the strongly connected case}. 
The reason is that given such branchings $T^-$ and $T^+$, it is then straightforward to schedule the arcs of $T^-$ from leaves to $r$, and then the arcs of $T^+$ from $r$ to leaves, to obtain an arc scheduling temporally connecting $|T^-|\cdot|T^+|$ pairs of vertices. If both branchings span a fraction $c$ of vertices, for some $c>0$, then this scheduling temporally connects at least a fraction $c^2$ of all pairs which guarantees approximation ratio at most $1/c^2$.

\paragraph*{Related work}
The undirected version of the \texttt{MRET} problem is studied in \cite{GobelCV91} where it is proven that it is NP-complete to decide whether an undirected graph has an edge scheduling connecting all pairs. However, one can easily produce an edge scheduling connecting a constant fraction of pairs by decomposing a spanning tree at a centroid $c$ so as to produce two disjoint trees $T,T'$ with the same root $c$ and each of them covering one third of the vertices. It is then straightforward to schedule edges so as to 
connect a constant fraction of pairs.
A similar (undirected) problem where an edge can be scheduled several times is considered in \lv{a series of papers \cite{AkridaGMS17,KlobasMMS22,MertziosMS19}}. 
The goal is then to minimize the total number of time labels used while respecting some constraints with respect to connnectivity or the maximum time label used in particular. Note that scheduling each edge twice is sufficient for temporally connecting every pair of nodes of a connected undirected graph \reviews{(a first set of labels can allow each vertex to reach the root of a spanning tree, while a second set of labels can allow the root to reach every node)}. 
The \texttt{MRET} problem can be seen as a simplified version of the problem of scheduling buses in a public transit network~\cite{BrunelliCV23}.
Related problems~\cite{EnrightMMZ21,MolterRZ21} target the minimization of temporally connected pairs and are driven by applications to mitigation of epidemic propagation.

It is shown in \cite{Bang-Jensen91} that it is NP-complete to decide whether a strongly connected digraph contains an in-branching and an out-branching with same root, which are edge disjoint, and that both span all vertices. The same paper also relates the conjecture that such branchings exist if the digraph is $c$-edge-connected for sufficiently large $c$. The conjecture holds for $c=2$ in digraphs with independence number at most 2~\cite{Bang-JensenBHY22}. It is also shown in \cite{Bang-JensenCH16} that it is NP-complete to decide whether a strongly connected digraph has a partition of its vertices into two parts of size at least 2, such that the first part is spanned by an in-branching and the second part is spanned by an out-branching.


\paragraph*{Main results}
Our approach is to address these problems from a vertex point of view, which is enough to obtain an approximation algorithm for \texttt{MRET}. Specifically, we consider two maximization problems: In the \texttt {Forward Connected Pairs Problem} (\texttt{FCPP}) the goal is to find an enumeration \reviews{(or ordering)} $v_1,\dots ,v_n$ of a strongly connected digraph $D=(V,A)$ such that the number of pairs $v_i,v_j$ joined by a directed path with increasing indices (called \emph{forward pairs}) is maximized. In the \texttt {Balanced Bi-Tree Problem} (\texttt{BBTP}) the goal is to find an in-tree and an out-tree only intersecting at their root (thus making a \emph{bi-tree}) with equal size (to be maximized). \reviews{(We use in-tree and out-tree instead of in-branching and out-branching to stress the requirement that they share no vertex apart from their common root, and hence do not span the digraph.)} \stb{Notice that the former problem, {\tt FCPP}, is NP-hard, as finding an enumeration joining $n(n-1)/2$ pairs of vertices with forward paths is equivalent to finding a Hamilton path in the instance (and Hamilton path in general digraphs can easily be reduced to Hamilton path in strong digraphs).}


The main result of this paper is that one can find a solution $T$ of \texttt{BBTP} in time $O(n^2)$ with size at least $n/3-1$. Therefore each in-tree and out-tree has size at least $n/6$, and by considering any enumeration of $V$ extending a topological ordering of $T$, we obtain a solution of \texttt{FCPP} of size at least $n^2/36$. Since the maximum possible solution of \texttt{FCPP} is $n^2/2$, this gives a $1/18$ approximation for the \texttt {Forward Connected Pairs Problem}. Our construction can be extended to a weighted digraph where each vertex $u$ is associated to a weight $w_u$. It then produces a bi-tree where the in-tree and the out-tree both have total weight at least $W/6$, where $W=\sum_{u\in V} w_u$ is the total weight of the digraph.

We also consider a covering version \texttt {CFCPP} of \texttt{FCPP} where we look for a minimal set of orderings of the vertex set such that for every pair $x,y$, one of $xy$ or $yx$ is a forward pair in one of the orderings. Since we can cover a positive fraction of pairs, it is natural to wonder if \texttt {CFCPP} \stb{always admits a solution with a constant number of orderings} (or at most $\log n$).
To this end, we consider a request variant of the problem, called \texttt{RFCPP}, where we ask to connect by forward paths a maximum number of pairs among a given set $R\subseteq {V\choose 2}$ of requested pairs. We provide a family of instances \stb{of \texttt{RFCPP} where the number of needed ordering to satisfy all the requests of $R$ is more than a constant fraction of $|R|$}. We do not know if \texttt {CFCPP} and \texttt {RFCPP} can be efficiently approximated. 
This still leaves open the existence of an $O(\log n)$ solution for \texttt {CFCPP}.
\lv{Note that both \texttt{CFCPP} and \texttt{RFCPP} are NP-hard for the same reason as \texttt{FCPP}.}
\lv{The approximation of \reviews{the variant of} \texttt{FCPP} \reviews{extended to} general digraphs is also left open. We think that solving the strong case is a key step towards this more ambitious goal since the acyclic case can easily be solved exactly \reviews{as mentioned previously}.}

\paragraph*{Main techniques: left-maximal DFS and balanced circuit separators}

From an algorithmic perspective, our solution relies on finding a \emph{cyclic balanced separator} $C$ of $D$. Specifically, the vertex set of $D$ is partitioned into three parts $I,C,O$ such that $C$ spans a circuit, no arc links a node from $I$ to a node in $O$, and which is balanced in the sense that both $I\cup C$ and $C\cup O$ have size at least $n/3$. Note that $I$ and $O$ can be empty, as this is the case when $D$ is a circuit.

Such a partition can be computed in linear time from a left-maximal depth-first-search (DFS) tree, that is a DFS tree such that the children of any node are ordered from left-to-right by non-increasing sub-tree size. Both of these structures could be of independent interest in the field of digraph algorithms. The computation of a left-maximal DFS is the (quadratic) complexity bottleneck of our algorithm. We feel that a linear-time algorithm for finding a cyclic balanced separator should be achievable either by other means, or by relaxing the requirement on left-maximal DFS (we just need it to be "not too much unbalanced to the right"). However, actually computing in linear time a left-maximal DFS tree could prove more challenging. Could decremental SCC help~\cite{BernsteinPW19}?

\section{Definitions}
\label{sec:definitions}

In this paper we consider directed graphs $D=(V,A)$ (\emph{digraphs}) in which cycles of length two are allowed. The set $V$ is the set of \emph{vertices} (usually $n$ of them) and $A$ is the set of \emph{arcs}. We say that $x,y$ are \emph{connected} in $D$ if there exists a directed path from $x$ to $y$. A digraph is \emph{strongly connected}, or simply \emph{strong}, if all $x,y$ are connected. In particular, if $D$ has $n$ vertices, the number of connected couples $x,y$ is $n^2$. The \emph{out-section} generated by a vertex $x$ of $D$ is the set of vertices $y$ for which $x,y$ are connected. We say that $x$ \emph{out-generates} $D$ if the out-section of $x$ is $V$. 

A \emph{schedule} is an injective mapping $f$ of $A$ into the positive integers. In a scheduled digraph $(D,f)$ we say that a couple of vertices $x,y$ is \emph{connected} if there is a directed path $x=x_1,x_2,\dots ,x_k=y$ such that $f(x_ix_{i+1})<f(x_{i+1}x_{i+2})$ for all $i=1,\dots,k-2$. We denote by $c(D,f)$ the number of connected couples and by $s(D)$ the maximum over all choices of $f$ of $c(D,f)$. We now define our scheduled ratio $r_s$ which is the infimum of $s(D)/n^2$ over strongly connected digraphs $D$ on $n$ vertices, when $n$ goes to infinity.

This ratio $r_s$ is based on scheduling arcs, and a similar ratio $r_t$ can be defined via a total order on vertices. We consider for this a digraph $D$ and a total ordering $<$ on its vertices. An arc $xy$ of $D$ is \emph{forward} if $x<y$. We say that a couple $(x,y)$ is \emph{connected} in $(D,<)$ if there is a directed path from $x$ to $y$ consisting of forward arcs, and we then call $(x,y)$ a \emph{forward couple}. We denote by $c(D,<)$ the number of forward couples and by $t(D)$ the maximum over all choices of $<$ of $c(D,<)$. The ordered ratio $r_t$ is the infimum of $t(D)/n^2$ over strongly connected digraphs $D$ on $n$ vertices, when $n$ goes to infinity.

The main problem we address in this paper is to show that both $r_s$ and $r_t$ are positive. Let us first prove that these questions are related.

\begin{theorem}\label{th:ratioequivalence}
$r_t\leq r_s$
\end{theorem}

\begin{proof}
Given $(D,<)$, we can consider that $<$ is a bijective mapping $g$ from $V$ to $0,\dots, n-1$ respecting the order (that is $x<y$ whenever $g(x)<g(y)$). Now observe that if one define $f(xy)=n.g(x)+g(y)$ for every arc, then every forward couple in $(D,<)$ is a connected pair in $(D,f)$.
\end{proof}


\noindent
Thus we can focus on the following problem.

\begin{problem}\label{pb:zeconjecture}
What is the value of $r_t$?
\end{problem}

The fact that $r_t>0$ is not obvious and is indeed our central result. Observe that we can assume that $D$ is \emph{minimally strongly connected}, i.e. every arc $xy$ of $D$ is the unique directed $x,y$-path in $D$. A classical result using ear-decompositions asserts that the number of arcs of a minimally strongly connected digraph is at most $2n-2$ \stb{(see~\cite{Bang-JensenGutinBook09} for instance)}. Unfortunately we were unable to use these decompositions to prove the positivity of $r_t$. Our strategy instead is to find inside $D$ a particular type of oriented tree.

An \emph{out-tree} $T^+$ is an orientation of a tree in which one \emph{root} vertex out-generates $T^+$. Reversing all arcs, we obtain an \emph{in-tree}. When identifying the root $r$ of an in-tree $T^-$ and the root of an out-tree $T^+$, we obtain a tree orientation called a \emph{bi-tree} $T$ where $r$ is the \emph{center}. Note that $x,y$ are connected for every $x\in T^-$ and $y\in T^+$. We say that a bi-tree is \emph{balanced} if $|T^+|=|T^-|$. In particular, if every strongly connected $D$ contains a balanced bi-tree of linear size, one directly obtains that $r_t>0$ for Problem~\ref{pb:zeconjecture}. 

\begin{problem}\label{pb:zeconjecture2}
What is the maximum $c_b$ for which every strongly connected directed graph on $n$ vertices has a balanced bi-tree of size at least $c_b.n$?
\end{problem}

We show in Theorem~\ref{th:bitree-existence} that $c_b\geq 1/3$, where both $|T^+|$ and $|T^-|$ have size at least $1/6$.
One can naturally ask if the enumeration problem directly implies the bi-tree problem. But the following example shows that this is not true in general.

\begin{proposition}
\label{prop:example-sec2}
For every integer $k$, there exists a minimally strongly connected digraph $D$ on $n=3k^2+2k+2$ vertices admitting an enumeration $E=v_1,\dots, v_n$ such that:

- the number of forward couples in $E$ is quadratic in $n$ (at least $k^4$), and

- the maximal size of a balanced bi-tree only using forward arcs in $E$ is $2k+5=O(\sqrt{n})$.
\end{proposition}


\begin{proof}
Let $D$ be the digraph with vertex set $\{x\}\cup A \cup A' \cup X \cup B' \cup B \cup \{y\}$ where $A'$, $B'$ have size $k$ and $A$, $B$ and $X$ all have size $k^2$. Consider $(A_i)_{1\le i\le k}$ a partition of $A$ into $k$ sets of size $k$ and $(B_i)_{1\le i\le k}$ a partition of $B$ into $k$ sets of size $k$. Moreover, we denote by 
$a_1,\dots a_k$ the vertices of $A'$, $b_1,\dots b_k$ the vertices of $B'$, and
$(x_{i,j})_{1\le i,j \le k}$ the vertices of $X$. Now, add to $D$ the following sets of arcs: 
$\{xa\ :\ a\in A\}$, 
$\{aa_{i}\ :\ a\in A_i\}$ for all $1\le i\le k$,
$\{a_ix_{i,j}\ :\ j=1,\dots ,k\}$ for all $1\le i\le k$,
$\{x_{i,j}b_j\ :\ i=1,\dots ,k\}$ for all $1\le j\le k$,
$\{b_{j}b\ :\ b\in B_j\}$ for all $1\le j\le k$,
$\{by\ :\ b\in B\}$ and $\{yx\}$. The construction is depicted in Figure~\ref{fig:Example}.\\

\begin{figure}[!ht]
\centering
\scalebox{0.8}{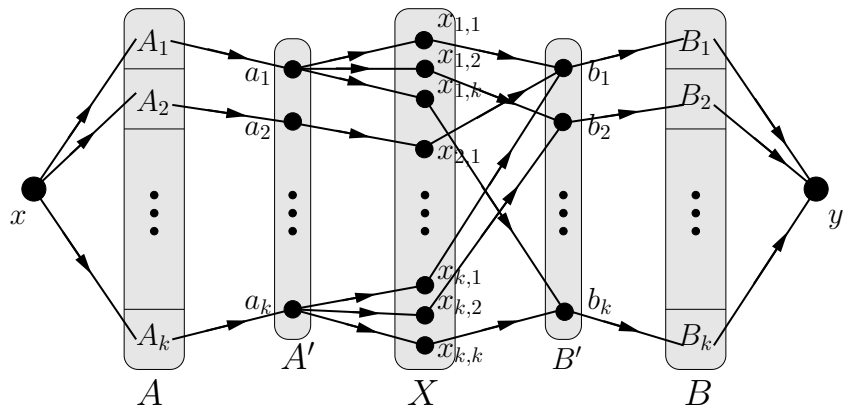}
\caption{The digraph $D$ in the proof of Proposition~\ref{prop:example-sec2}. An arc between a block and a particular vertex stands for all the arcs between \reviews{each vertex} of the block and the particular vertex. The arc $yx$ is not drawn. 
}
\label{fig:Example}
\end{figure}

The digraph $D$ is strongly connected and has $n=3k^2+2k+2$ vertices. Furthermore, $D$ is minimally strongly connected, as for every arc $uv$ we have either $d^+(u)=1$ (if $uv=yx$ or $u\in A\cup X \cup B$) or  $d^-(v)=1$ (if $v\in A\cup X\cup B$).

Consider now any enumeration of $D$ where $x$ is the first vertex, then $A$ is before $A'$, then $A'$ is before $X$, then $X$ is before $B'$, then $B'$ is before $B$, and finally $y$ is the last vertex. For such an enumeration all the arcs of $D$ are forward except $yx$. For every $1\le i,j\le k$, any vertex of $A_i$ has a path to any vertex of $B_j$ using the vertex $x_{ij}$. So the number of forward couples is at least $|A|.|B|=k^4$, which is quadratic in $n$. However, the largest balanced bi-tree only using forward arcs has its  center in $X$ and has size $2k+5=O(\sqrt{n})$. \reviews{Indeed, any node $x_{i,j}\in X$ has only one in-neighbor which has itself $k$ in-neighbors, and only one out-neighbor which has itself $k$ out-neighbors, resulting in $2k+3$ nodes which can further connect to $x$ and $y$ only.}
\end{proof}



\section{Computing a left-maximal depth first search tree}
\label{sec:bitree-existence}

Given an out-tree $T$ and a node $x$ of $T$, we denote by $T_x$ the \emph{subtree} rooted at $x$ consisting of all vertices in the out-section of $x$ in $T$. A \emph{child} of $x$ is an out-neighbor of $x$ in $T$.
Let $D=(V,A)$ be a directed graph. A \emph{depth first search tree} $T$ of $D$ (\emph{\textsc{dfs}-tree} for short) is an out-tree which is a spanning subgraph of $D$ with the following properties:

\begin{itemize}
    \item For every node $x$, there is a total order $\leq _x$ on the children of $x$, which is called a \emph{left-right order}.
    \item If $y,z$ are two children of $x$ and there exists an arc of $D$ from $T_z$ to $T_y$, then $y\leq _x z$ (i.e. arcs between disjoint subtrees goes from right to left).
\end{itemize}

We are interested in a particular type of \textsc{dfs}-tree called \emph{left-maximal} in which we have $|T_y|\geq  |T_z|$ for every node $x$ and children $y,z$ such that $y\leq _x z$. In other words the size of the child subtrees of any vertex is non-increasing from left to right. 

\begin{theorem}\label{th:leftdfs-existence}
If $x$ out-generates $D$, then there exists a left-maximal \textsc{dfs}-tree $T$ rooted at $x$. Moreover $T$ can be computed in quadratic time in minimally strongly connected digraphs.
\end{theorem}

\begin{proof}
  We construct a left-maximal \textsc{dfs}-tree rooted at $x$ as follows. Compute the strongly connected components of $D\setminus\{x\}$ and their sizes. Consider the acyclic digraph $D'$ between components where an arc $(C,C')$ indicates that there is an arc from $C$ to $C'$  in $D$. By traversing $D'$ according to a reverse topological order, we obtain the size of the out-section of each node in $D\setminus\{x\}$. Let $y$ be a node with out-section $S_y$ having maximum size. Note that $y$ belongs to some strongly connected component $C$ which is a source in $D'$. Since $D$ is strongly connected, $C$ contains an out-neighbor of $x$. Free to choose $y$ in $C$, we assume for simplicity that $xy$ is indeed an arc. Construct recursively a left-maximal \textsc{dfs}-tree $T'$ rooted at $x$ in $D\setminus S_y$, and a left-maximal \textsc{dfs}-tree $T''$ rooted at $y$ in the digraph induced by $S_y$. The final out-tree $T$ is obtained by inserting $T''$ as the leftmost child of $x$ in $T'$. It is indeed a \textsc{dfs}-tree as there exists no arc from $S_y$ to any node in $D\setminus T'$ by the definition of outsection. To realize that it is also left-maximal, consider the leftmost child $z$ of $x$ in $T'$. The outsection $S_z$ of $z$ in $D\setminus\{x\}$ has size at most $|S_y|$, and we thus have $|T_z|\le |S_z|\le |S_y|=|T_y|$.

  As the computation of strongly connected components, digraph $D'$, and outsection sizes can be done in linear time, the whole computation can be done in $O(n^2)$ time in minimally strongly connected graphs (where the number of arcs is linear in the number of vertices). 
  \end{proof} 

\begin{problem}\label{pb:calculdfs}
Can we compute a left-maximal \textsc{dfs}-tree in $o(n^2)$ time in minimally strongly connected digraphs? What about the complexity in the general case?
\end{problem}

\medskip
Let us now introduce our key-definition which is a particular type of partition of a strongly connected digraph $D$. To get beforehand a bit of intuition, one can picture a connected (undirected) graph $G$ with a depth first search tree $T$ drawn on the plane. The key-feature here is that any path $P$ from the root to a leaf partitions the rest of the graph into two subsets $L$ and $R$ which are respectively the vertices of $V\setminus P$ to the left and to the right of $P$. In other words, $G$ has a  cutset $V(P)$ \reviews{(i.e. its removal splits the graph into several connected components)} with a remarkable property since it is spanned by a path. Observe also that \emph{any} root-leaf path can be used, so by a classical left-right sweeping argument, one can find $P$ such that both $L$ and $R$ have size at most $2n/3$ (so that the cut is balanced). We now generalize this argument to strongly connected digraphs, where a directed cycle takes over the role of $P$.

\smallskip
An \emph{$(I,C,O)$-decomposition} of a strongly connected digraph is defined as:

\begin{itemize}
    \item a partition of $V$ into three subsets $I,C,O$,
    \item $C$ is spanned by a directed cycle, 
     \item there is no arc from $I$ to $O$.
\end{itemize}

Observe that, by strong connectivity and the fact that there is no arc from $I$ to $O$, for every vertex $x$ in $I$, there exists a directed path from $x$ to $C$ with internal vertices inside $I$. Similarly,
for every vertex $x$ in $O$, there exists a directed path from $C$ to $x$ with internal vertices inside $O$. We now show how to get a  $(I,C,O)$-decomposition \reviews{which is additionally \emph{balanced}, that is such that both $I\cup C$ and $O\cup C$ have size greater than $n/3$}.

\smallskip
Given a \textsc{dfs}-tree $T$, we call \emph{left path} $L_T$ the path starting at the root  of $T$ and which iteratively selects the leftmost child of the current vertex as the next vertex of the path. In particular, in a planar drawing respecting the left-right order of children, $L_T$ is the path from the root to the leftmost leaf.
Let $x$ be a vertex of the left path of $L_T$. 
Given a child $y$ of $x$, we define the subtree $T_{x,y}$ of $T$ as the subtree of $T_x$ containing $x$ and all $T_z$ for $z\leq_x y$. We call $T_{x,y}$ a \emph{left subtree} of $T$. Note that $T_x$ is the left subtree $T_{x,y}$ obtained by selecting $y$ as its rightmost child. Also, if $y$ is the leftmost child, $x$ has only one child in $T_{x,y}$.

\begin{proposition}\label{prop:IOC-dec}
For every left subtree $T_{x,y}$, there exists an $(I,C,O)$-decomposition such that $T_{x,y}$ is included in $I\cup C$ and $V\setminus T_{x,y}$ is included in $O\cup C$.
\end{proposition}

\begin{proof}
As $T$ is a \textsc{dfs}-tree, any arc outgoing from $T_{x,y}\setminus \{x\}$ reaches a vertex in $T_{x,y}$ or in $L_T$. 
By strong connectivity, some arc $uv$ with $u\in T_{x,y}\setminus \{x\}$ reaches some vertex in $L_T$ between the root and $x$ \reviews{(possibly the root or $x$)}, and we select $uv$ such that $v$ has minimum distance from the root in $T$ \reviews{(see Figure~\ref{fig:ico} (Left) for an example)}. Note that $v$ can be equal to $x$ \reviews{(some arc $u'v'$ with $u'\in T_x\setminus T_{x,y}$ must then lead to a vertex $v'\in L_T$ closer to the root by strong connectivity)}. We define $C$ as the cycle formed the path $P_{vu}$ from $v$ to $u$ in $T$
and the arc $uv$.
We now set $I:=V(T_{x,y})\setminus C$ and $O=V\setminus (C\cup I)$. Every arc leaving $T_{x,y}\setminus \{x\}$ reaches a vertex $w$ of the left-path between $v$ and $x$ by the choice of $v$. Thus $w$ is in $C$, and hence there is no arc from $I$ to $O$.
\end{proof}

\begin{figure}
    \centering
    \includegraphics[clip, trim=0.8mm 0.8mm 0.8mm 0.6mm, width=.49\textwidth]{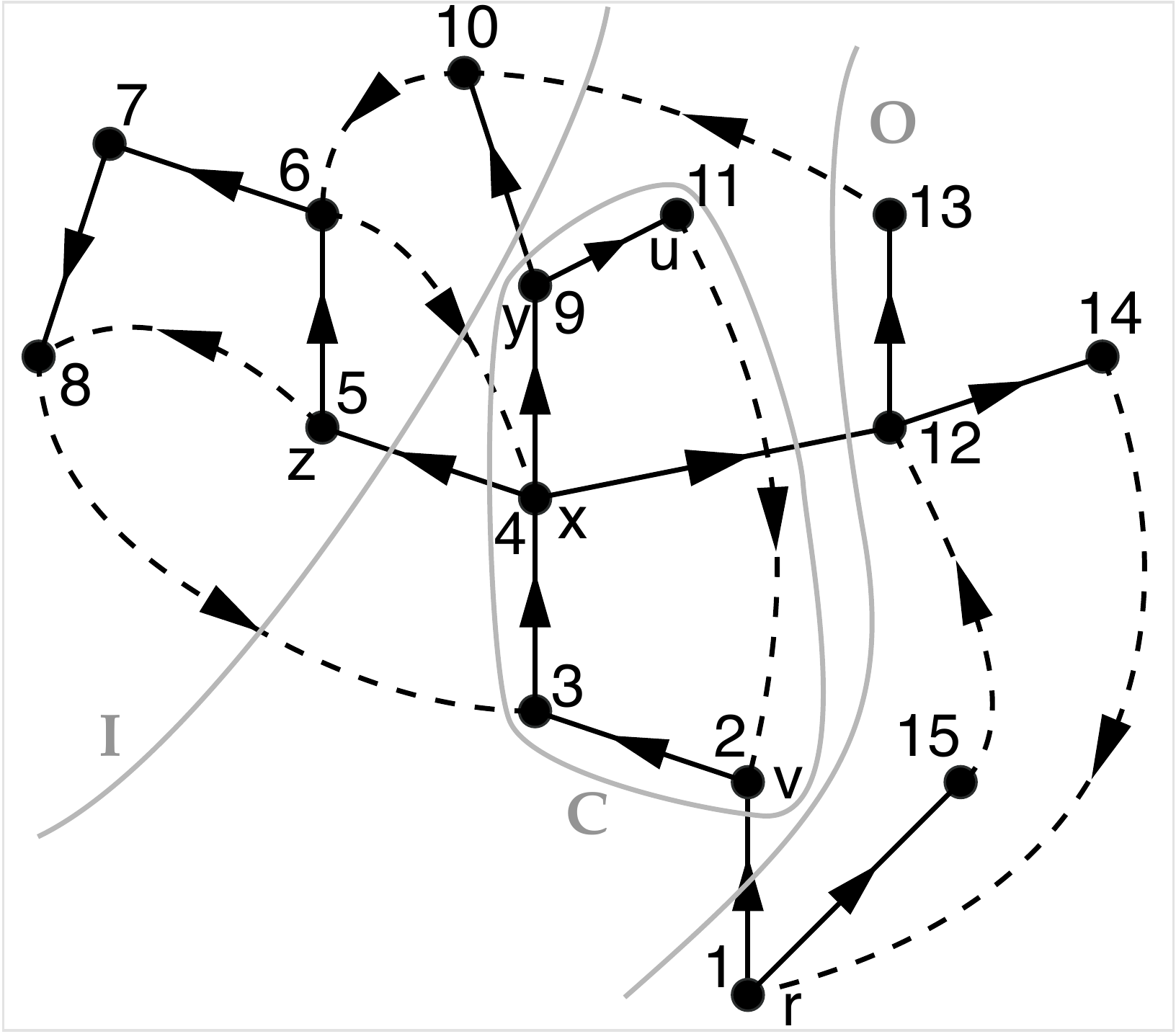}\hfill
    \includegraphics[clip, trim=0.8mm 0.8mm 0.8mm 0.6mm, width=.49\textwidth]{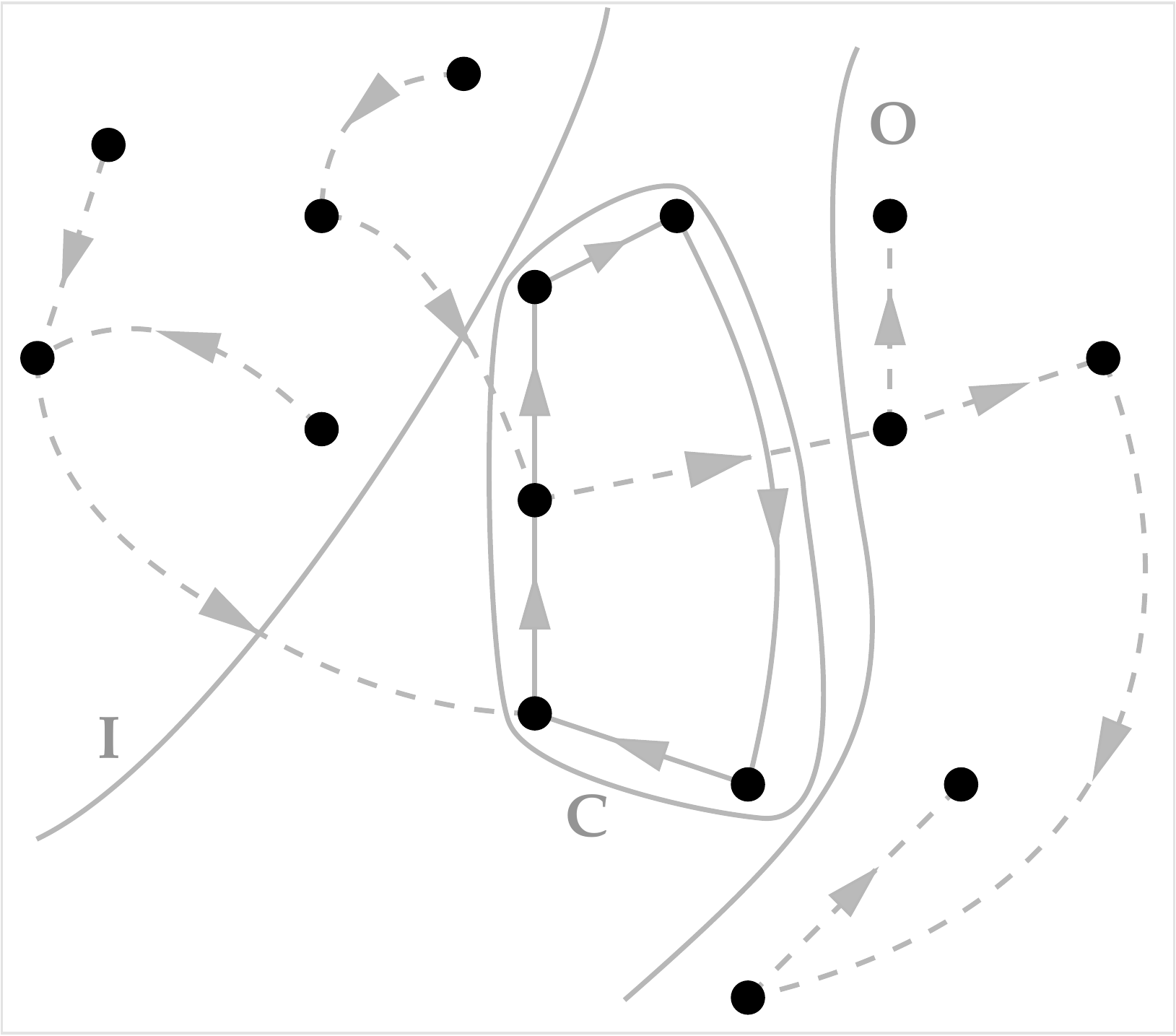}
    \caption{Left: a strong digraph (all arcs) with a left-maximal \textsc{dfs}-tree $T$ (plain arcs) and the $(I,C,O)$ decomposition associated to $T_{x,y}$. Nodes are numbered according to the corresponding \textsc{dfs} traversal. Right: the cycle spanning $C$ (plain arcs), two in-trees spanning $I$ and one out-tree spanning $O$.}
    \label{fig:ico}
\end{figure}

\begin{proposition}\label{prop:balancedtree}
Every strongly connected directed graph with $n\ge 4$ vertices has a \textsc{dfs}-tree $T$ which has a left subtree $T_{x,y}$ such that $n/3< |T_{x,y}|< 2n/3$.
\end{proposition}
\begin{proof}
We consider a left-maximal \textsc{dfs}-tree $T$ of $D$. Scanning the left-path of $T$ from the root, we consider the first node $z$ such that $|T_z|\leq  n/3$. Let $x$ be the parent of $z$. If $|T_z|= n/3$, the left subtree $T_{x,z}$ with size $n/3+1$ is the solution we are looking for. Let $y$ be the rightmost child of $x$ such that $|T_{x,y}|<2n/3$. Assume for contradiction that $|T_{x,y}|\leq n/3$. By the definition of $z$, we have $|T_x|>n/3$, hence $y$ is not the rightmost child of $x$. In particular, $y$ has a (next) right sibling $y'$. As $|T_{y'}|\leq |T_z|<n/3$ by left-maximality, we reach a contradiction to the choice of $y$ since $|T_{x,y'}|<2n/3$. Thus $T_{x,y}$ is our solution. 
\end{proof}


\begin{corollary}\label{cor:balanceddec}
Every strongly connected directed graph has a balanced $(I,C,O)$-decomposition.
(i.e. both $I\cup C$ and $O\cup C$ have size strictly greater than $n/3$).
\end{corollary}
\begin{proof}
If $n\ge 4$, then this is a direct consequence of Proposition~\ref{prop:IOC-dec} and Proposition~\ref{prop:balancedtree}. The case $n\leq 3$ follows by enumerating the cases.
\end{proof}

We now compute a linear size bi-tree from it.

\section{Bi-labels}
\label{sec:bilabel}

We now consider a directed graph $D$ equipped with a \emph{bi-label}, that is every vertex $x$ receives a couple $(i(x),o(x))$ of positive integers. The \emph{weight} of $D$ is $(i(V),o(V))$, the sum of all $i$ and $o$ values respectively. Assume that $D$ is a bi-labelled digraph which is the union of a digraph $D'$ and an out-tree $T^+$ rooted at $r\in V(D')$ such that $D'\cap T^+=\{r\}$. We can \emph{transfer} the weight $o(T^+)$ to $D'$ by 
adding $o(T^+\setminus r)$ to $o(r)$
and removing all the vertices of $T^+\setminus r$. We define similarly the transfer operation of $i$ for an in-tree.

Given a bi-tree $B$ of $D$, the \emph{value} of $B$ is the pair $(a,b)$ where $a$ is the sum of all $i(x)$ for $x$ in $B^-$ and $b$ is the sum of all $o(x)$ for $x$ in $B^+$. In other words, the value of $B$ is the label of the center after the transfers of $B^+$ and $B^-$. Observe that if $D'$ is obtained from $D$ by some transfer and $D'$ has a bi-tree with value $(a,b)$, then $D$ also has a bi-tree with value $(a,b)$.

\begin{theorem}\label{th:bilabelcycle}
If $C$ is a cycle equipped with a bilabel $(i,o)$ of weight $(w,w)$, it contains a bi-tree with value at least $(w/2,w/2)$. 
\end{theorem}

\begin{proof}
Consider a shortest path $P$ in $C=x_1,\dots ,x_n$ such that $i(P)\geq w/2$ or $o(P)\geq w/2$. Assume without loss of generality that $P=C[x_1,\dots ,x_k]$. First consider the case where we have $i(P)\geq w/2$. By minimality of $P$, we have $o(C[x_k,\dots ,x_n])\geq w/2$, and therefore the bi-tree $B$ centered at $x_k$ such that $B^+=C[x_k,\dots ,x_n]$ and $B^-=C[x_1,\dots ,x_k]$ satisfies the hypothesis. In the case $o(P)\geq w/2$, we proceed similarly with $x_1$ as center.
\end{proof}

The bound in Theorem~\ref{th:bilabelcycle} is sharp. Consider for this a directed 4-cycle in which all labels are $(w/4,w/4)$: any bi-tree has value $(a,b)$ with $\min\{a,b\}=w/2$.




\begin{theorem}\label{th:bitree-existence}
Every strong digraph $D=(V,A)$ on $n$ vertices contains a bi-tree $B$ such that both $B^+$ and $B^-$ have size at least $n/6$.
\end{theorem}

\begin{proof}
By Corollary~\ref{cor:balanceddec}, $D$ has an $(I,C,O)$-decomposition such that both $I\cup C$ and  $O\cup C$ have size at least $n/3$.
Observe that $D$ is spanned by a subgraph $S$ consisting of the directed cycle spanning $C$, together with a disjoint collection of in-trees rooted at some vertices of $C$ and with other vertices in $I$, and a collection of out-trees rooted at some vertices of $C$ and with other vertices in $O$ \reviews{(see Figure~\ref{fig:ico} (Right) for an example)}.  
\reviews{This comes from strong connectivity which implies that any node $u$ in $I$ has at least one out-neighbor $v$ and $v$ must be either in $I$ or in $C$ by the definition of the $(I,C,O)$-decomposition which forbids any arc from $I$ to $O$. Selecting arbitrarily one out-neighbor for each node in $I$ results in a collection of in-trees rooted at vertices in $C$ and spanning $I$. Similarly, selecting arbitrarily one in-neighbor for each node in $O$ results in a collection of out-trees rooted at vertices in $C$ and spanning $O$.}
We consider that $S$ is a bi-labelled digraph by setting the value $(1,1)$ to every vertex. We can then transfer the weight of all in-trees and out-trees to their respective roots to obtain a bi-label on $C$ with weight at least $(n/3,n/3)$. We now invoke Theorem~\ref{th:bilabelcycle}
to obtain a bi-tree with value at least $(n/6,n/6)$ for the cycle, and unfold it by reversing appropriate transfers to get a bi-tree of $S$ with same value.
\end{proof}

\stb{As mentioned in the introduction, from the previous results we obtain the following .

\begin{corollary}
{\tt FCCP} admits a 1/18-approximation in quadratic time.
\end{corollary}
}

\medskip
Note that the above constructions of left-maximal \texttt{DFS} and bi-tree can be extended to a weighted digraph $D$ where each vertex $u$ has a non-negative weight $w_u$.
Given a subset $U\subseteq V$ of $n'$ vertices, we can set $w_u=1$ for $u\in U$ and $w_u=0$ for $u\notin U$ to compute similarly a bi-tree $B$ such that $B^-$ and $B^+$ both span $n'/6$ vertices of $U$. One can then easily obtain an ordering of $V$ such that a constant fraction of pairs in $U\times U$ are forward-connected.


We now consider the more general version of the problem where we want to connect pairs in a given set $R\subseteq {V\choose 2}$ of requests.

\section{Forward connecting a set of requested pairs}
\label{sec:request}

 In the {\tt Request Forward Connected Pairs Problem} ({\tt RFCPP}), the input is a strongly connected digraph $D=(V,A)$ and a set $R\subseteq {V\choose 2}$ of \emph{requests}, and the output is a vertex ordering of $D$ maximizing the number of \emph{forward pairs} $\{x,y\}\in R$, that is unordered pairs $\{x,y\}$ such that either $x,y$ or $y,x$ is forward. Note that {\tt FCPP} is the particular instance of {\tt RFCPP} satisfying $R={V\choose 2}$. Since {\tt FCPP} is in APX, it is natural to raise the following problem:

 \begin{problem}\label{pb:forwrequest}
Is there a polytime constant approximation algorithm for {\tt RFCPP}?
\end{problem}

In particular, is it always possible to satisfy a linear fraction of $R$? In the more restricted variant where $R$ is a set of couples $(x,y)$ instead of pairs (where one wants to maximize the number of forward couples $x,y$), one cannot expect to satisfy a large proportion of $R$. Indeed, if $D$ is the circuit $(v_1,\dots ,v_n)$ and $R$ consists of all couples $(v_{i+1},v_i)$, any vertex ordering can only satisfy at most one request of $R$, hence only a ratio of $1/n$ can be realized as forward couples.

Surprisingly, when requests are pairs, we could not find any set of request $R$ which is not satisfied within a ratio of $1/\log n$. However, contrary to the case of {\tt FCPP} where a positive ratio is achievable, the following result shows that there are instances of {\tt RFCPP} for which only a logarithmic ratio can be realized as forward pairs.


\begin{proposition}\label{prop:ratiorequest}
For all $n$, there exists an instance of  {\tt RFCPP} where $D$ has $2^{n+1}-1$ vertices, $R$ has size $n2^{n-1}$, and no more than $2^n$ requested pairs can be realized as forward paths.
\end{proposition}

\begin{proof}
The graph $D$ is the complete binary tree of height $n$ seen as a directed graph by letting each edge to be a circuit of length two.
The set of requests is recursively defined in the following way. We consider the set $L_l$ of all leaves which are descendants of the left child of the root $r$ and the set $L_r$ of leaves descendants of its right child. Now we pick an arbitrary perfect matching $M$ between $L_l$ and $L_r$ and set as requests all the $|M|$ pairs formed by the edges of $M$. We call this set $R_1$, and recursively define in the same way a set of requests for each of the two children of $r$. We stop when reaching leaves. With the right choice of matchings, the set of requests can correspond for instance to a hypercube of dimension $n$ on the leaves. We shall not need it, but it gives a good intuition of the possible structure of requests.

To sum up, $D$ has $2^{n+1}-1$ nodes, $2^n$ of them being leaves. Each internal node creates a matching of requests in its set of descendant leaves, hence the total number of requests is $n.2^{n-1}$.

Any ordering $<$ of $V(D)$ can be considered as the sub-digraph $D'$ of $D$ where we keep only arcs $xy$ satisfying $x<y$. Note that this corresponds to an orientation of the edges of the tree.
Our goal is to show that the number of requests $x,y$ which are connected by a forward path is at most $2^n$. To show this, for a given node $x$ we denote by $\rf(x)$ the number of requests between descendants of $x$ (in the tree $D$) which are realized by a forward path. We also denote by $in(x)$ the number of leaves $y$ descendant of $x$ such that there is a forward path from $y$ to $x$. Finally $out(x)$ is the number of leaves $z$ descendant of $x$ such that there is a forward path from $x$ to $z$.

We now show by induction that both $\rf(x)+in(x)$ and $\rf(x)+out(x)$ are upper bounded by $\ell (x)$, the number of leaves descendant of $x$. This is true if $x$ has two leaves as children since either $x$ is a source (resp. a sink) in $D'$ and $\rf(x)+\max(in(x),out(x))=0+2$, or $x$ has both in and out-degree 1 and $\rf(x)+in(x)=\rf(x)+out(x)=1+1$. For the induction step, we assume that $x$ has two children $y,z$.
\begin{itemize}
    \item if we have both arcs $xy,xz$ in $D'$, we have $\rf(x)+in(x)\leq \rf(x)+out(x)=\rf(y)+\rf(z)+out(y)+out(z)$, which by induction is at most $\ell(y)+\ell(z)=\ell(x)$. \item if we have both arcs $yx,zx$, we conclude similarly.
    \item if we have both arcs $yx,xz$, note that $\rf(x)\leq \rf(y)+\rf(z)+\min (in(y),out(z))$. Also $in(x)=in(y)$ and $out(x)=out(z)$. Thus, if $in(y)\leq out(z)$ we have $\rf(x)+in(x)\leq \rf(x)+out(x)\leq \rf(y)+\rf(z)+\min (in(y),out(z))+out(z)=\rf(y)+in(y)+\rf(z)+out(z)\leq \ell(y)+\ell(z)=\ell(x)$. And the same conclusion holds when $in(y)>out(z)$.
    \item if we have both arcs $xy,zx$, we conclude as previously.
\end{itemize}

Thus the maximum number of forward requests is $2^n$. 
\end{proof}

Noteworthily, for the instance in Proposition~\ref{prop:ratiorequest}, there is a set of $2n$ orderings such that every pair $x,y$ of leaves is forward connected in one of the orderings. Let us call a \emph{forward cover} of $D$ a set of vertex orderings $\leq_1,\dots ,\leq_k$ such that for every pair $\{x,y\}$, there is some $i$ such that $x,y$ or $y,x$ is a forward pair in $\leq _i$. This suggests the following conjecture:

\begin{conjecture}\label{conj:covering}
Every strong digraph $D$ on $n$ vertices has a $O(\log n)$ size forward cover.
\end{conjecture}

Conjecture~\ref{conj:covering} holds when $D$ is a bi-oriented graph\reviews{, that is an undirected graph considered as a digraph by replacing each edge $\{u,v\}$ by two arcs $uv$ and $vu$}. Here is a sketch: it suffices to consider a spanning tree $T$ of $D$ and to fix a centroid $c$. We then let $T=T_1\cup T_2$ where $T_1\cap T_2=\{c\}$ and both $T_1,T_2$ have size at least $n/3$. We consider any vertex enumeration $T_1\leq c\leq T_2$. Now we find recursively in $T_1$ and $T_2$ two families $F_1$ and $F_2$ of vertex enumerations of size logarithmic in $2n/3$. We conclude by gluing (on $c$) the orderings in $F_1$ and $F_2$ by pairs. We get in total $1+\max(|F_1|,|F_2|)$ orders, hence a logarithmic size family.

Another intriguing question is the "forward orientation" of a pair $x,y$. We have seen that in some cases, like the circuit, pairs $v_i,v_{i+1}$ are (obviously!) much more likely to be forward than pairs $v_{i+1},v_i$. If indeed a forward cover of logarithmic size exists, it could be that some couples $x,y$ are more involved in a small cover than their reverse $y,x$. This suggests a kind of "forwardness" of a pair of vertices which might be interesting to characterize.

\section{Questions} 
\label{sec:questions}

Our main open question concerns Conjecture~\ref{conj:covering}. A  possible way to solve it would be to find $O(\log n)$ bi-trees such that for any pair $\{x,y\}$, there exists a bi-tree $T$ such that either $x\in T^-$ and $y\in T^+$, or $y\in T^-$ and $x\in T^+$. The reason is that each bi-tree can easily be converted to an ordering where each couple $(x,y)$ with $x\in T^-$ and $y\in T^+$ is a forward couple.
\lv{Note that Conjecture~\ref{conj:covering} would imply that any instance of \texttt{RFCPP} always have a solution realizing a $1/O(\log n)$ fraction of requested pairs, opening the possibility for polynomial time $O(\log n)$-approximation.}

Another question concerns the maximum size of a balanced bi-tree that can be found in any strongly connected digraph with $n$ vertices. Our construction shows that any such digraph contains a balanced bi-tree where both trees have size $n/6$. The construction given in \cite{BalliuBCOV2023} implies that there exists strongly connected digraphs such that in any bi-tree $T$, either $T^-$ or $T^+$ has size at most $n/3+O(1)$. This leaves a factor 2 gap.

If we relax the bi-tree definition by requiring that the \reviews{in-branching and the out-branching are edge disjoint (an \emph{in-out-branching}) and may overlap over more than one vertex (and still share the same root). What is the maximum size of a balanced in-out-branching} in any strongly connected digraph? The upper-bound of $n/3+O(1)$ given in \cite{BalliuBCOV2023} indeed holds for in-out-branchings. Is there a gap between the maximum size of a balanced in-out-branching and that of a balanced bi-tree?

More generally, what are the exact values of $r_s$ and $r_t$? In other words, what is the maximum ratio of couples that can be connected through an ordering of the arcs, or an ordering of the vertices respectively? Is it possible to obtain an interesting lower-bound of $r_t$ as a function of $r_s$?

\subsubsection*{Acknowledgement}
This research was partially supported by the 
ANR project Digraphs
ANR-19-CE48-0013 and the ANR project Tempogral ANR-22-CE48-0001.

\bibliographystyle{plain}
\bibliography{biblio}

\end{document}